\documentclass{mrlart7}
  \def\volno{18}
  \def\yrno{2011}
  \def\issueno{00}
  \def\lpageno{100NN} 
\usepackage{calc}
\usepackage{graphicx}
\usepackage{amsmath,amssymb,epsfig}
\usepackage{amsthm}
\usepackage{enumerate}

\newcommand{\N}{{\mathbb N}}

\newcommand{\Z}{{\mathbb Z}}
\newcommand{\Q}{{\mathbb Q}}

\newcommand{\T}{{\mathcal T}}

\newcommand{\F}{{\mathcal{F}}}

\newcommand{\Sel}{{\mathrm{Sel}}}
\newcommand{\dimF}{{\mathrm{dim}_{\mathbb{F}_2}}}
\newcommand{\Ftwo}{{\mathbb{F}_2}}

\newcommand{\Zt}{{\mathbb{Z}/2\mathbb{Z}}}

\newcommand{\hatphi}{{ \hat \phi }}

\newcommand{\p}{{ \mathfrak{p} }}

\newcommand{\ord}{{ \mathrm{ord} }}

\newcommand{\En}{{E_{(n)}}}

\newtheorem{thm}{\bf{Theorem}}
\newtheorem{theorem}{\bf{Theorem}}[section]

\newtheorem{proposition}[theorem]{\bf{Proposition}}

\newtheorem{lemma}[theorem]{\bf{Lemma}}

\theoremstyle{definition}

\newtheorem{remark}[theorem]{\bf{Remark}}

\input xy 
\xyoption{all} 

\usepackage[OT2,T1]{fontenc}
\DeclareSymbolFont{cyrletters}{OT2}{wncyr}{m}{n}
\DeclareMathSymbol{\Sha}{\mathalpha}{cyrletters}{"58}

\address{Department of Mathematics, University of Wisconsin-Madison, 480 Lincoln Drive, Madison, Wisconsin 53706}
\email{klagsbru@math.wisc.edu}

\begin{document}



\renewcommand{\baselinestretch}{1.5} \small\normalsize    

\title[\tiny{Elliptic Curves with a Lower Bound on 2-Selmer Ranks of Quadratic Twists}]{Elliptic Curves with a Lower Bound on 2-Selmer Ranks of Quadratic Twists}
\author{Zev Klagsbrun}

\begin{abstract}
For any number field $K$ with a complex place, we present an infinite family of elliptic curves defined over $K$ such that $\dimF \Sel_2(E^F/K) \ge \dimF E^F(K)[2] + r_2$ for every quadratic twist $E^F$ of every curve $E$ in this family, where $r_2$ is the number of complex places of $K$. This provides a counterexample to a conjecture appearing in work of Mazur and Rubin.
\end{abstract}

\maketitle

\pagenumbering{arabic}

\section{Introduction}

\subsection{Distributions of Selmer Ranks}

Let $E$ be an elliptic curve defined over a number field $K$ and let $\Sel_2(E/K)$ be its 2-Selmer group (see Section \ref{bg} for its definition). The \textbf{2-Selmer rank} of $E$, denoted $d_2(E/K)$, is defined as $$d_2(E/K) =  \dimF  \Sel_2(E/K) - \dimF E(K)[2].$$

For a given elliptic curve and positive integer $r$, we are able to ask whether $E$ has a quadratic twist with $2$-Selmer rank equal to $r$. A single restriction on which $r$ can appear as a $2$-Selmer rank within the quadratic twist family of a given  curve $E$ is previously known. Using root numbers, Dokchitser and Dokchitser identified a phenomenon called \textbf{constant $2$-Selmer parity} where $d_2(E^F/K) \equiv d_2(E/K) \pmod{2}$ for every quadratic twist $E^F$ of $E$ and showed that $E$ has constant 2-Selmer parity if and only if $K$ is totally imaginary and $E$ acquires everywhere good reduction over an abelian extension of $K$ \cite{DD2}.

In this paper, we show the existence of an additional obstruction to small $r$ appearing as 2-Selmer ranks within the quadratic twist family of $E$. We prove that there are curves having this obstruction over any number field $K$ with a complex place. Specifically:

\begin{thm}\label{badfamily}
For any number field $K$, there exist infinitely many elliptic curves $E$ defined over $K$ such that $d_2(E^F/K) \ge r_2$ for every quadratic $F/K$. Moreover, these curves do not have constant 2-Selmer parity and none of them become isomorphic over $\overline{K}$.
\end{thm}

This result disproves a conjecture appearing in \cite{MR} which predicted that subject only to the restriction of constant 2-Selmer parity, the set of twists of $E$ having 2-Selmer rank $r$ has positive density within the set of all twists of $E$ for every $r \ge 0$.


We prove Theorem \ref{badfamily} by presenting a family of elliptic curves defined over $\Q$ for which each curve in the family has the appropriate property when viewed over $K$. For $n \in \N$, let $E_{(n)}$ be the elliptic curve defined by the equation \begin{equation}\label{modelEn}E_{(n)}: y^2 + xy = x^3 - 128n^2x^2 - 48n^2x - 4n^2\end{equation} and define $\mathcal{F}$ as $\mathcal{F} = \{ E_{(n)} :  n \in \N,  1+256n^2 \not \in (K^\times)^2  \}.$ Each curve $E \in \mathcal{F}$ has a single point of order $2$ in $E(K)$ and a cyclic 4-isogeny defined over $K(E[2])$ but not $K$. Let $\phi:E \rightarrow E^\prime$ be the isogeny whose kernel is $C = E(K)[2]$. Our results are obtained by using local calculations combined with a Tamagawa ratio of Cassels to establish a lower bound on the rank of the Selmer group associated to $\phi$ (to be defined in Section \ref{bg}).


While curves $E \in \F$ have the property that $d_2(E^F/K) \ge r_2$ for every quadratic $F/K$, this does not hold in general for curves $E$ with $E(K)[2] \simeq \Zt$ that have a cyclic 4-isogeny defined over $K(E[2])$ but not over $K$. In particular, forthcoming work of this author can be used to show that every $r \ge 0$ appears infinitely often as a 2-Selmer rank within the quadratic twist family of $E^\prime$ for every $E \in \mathcal{F}$ \cite{K}.


\section{Selmer Groups}\label{bg}

We begin by briefly recalling the constructions of the 2-Selmer and $\phi$-Selmer groups along with some of the standard descent machinery. A more detailed explanation can by found in Section X.4 of \cite{AEC}.

If $E$ is an elliptic curve defined over a field $K$, then the Kummer map $\delta_{[2]}$ maps $E(K)/2(K)$ into $H^1(K, E[2])$. 
If $K$ is a number field, then for each place $v$ of $K$ we define a distinguished local subgroup $H^1_f(K_v, E[2]) \subset H^1(K_v, E[2])$ by \begin{equation*}\text{Image} \left (\delta_{[2]}: E(K_v)/2E(K_v) \hookrightarrow H^1(K_v, E[2]) \right ).\end{equation*} We define the \textbf{2-Selmer group} of $E$, denoted $\Sel_2(E/K)$, by $$\Sel_2(E/K) = \ker \left ( H^1(K, E[2]) \xrightarrow{\sum res_v} \bigoplus_{v\text{ of } K} H^1(K_v, E[2])/H^1_f(K_v, E[2]) \right ).$$



If $E^F$ is the quadratic twist of $E$ by $F/K$ where $F$ is given by $F = K(\sqrt{d})$, then there is an isomorphism $E \rightarrow E^F$ given by $(x, y) \mapsto (dx, d^{3/2}y)$ defined over $F$. Restricted to $E[2]$, this map gives a canonical $G_K$ isomorphism $E[2] \rightarrow E^F[2]$, allowing us to view $H_f^1(K_v, E^F[2])$ as sitting inside $H^1(K_v, E[2])$. The following lemma due to Kramer describes the connection between $H^1_f(K_v E[2])$ and $H^1_f(K_v, E^F[2])$. 

Given a place $w$ of $F$ above a place $v$ of $K$, we get a norm map $E(F_w) \rightarrow E(K_v)$, the image of which we denote by $E_\mathbf{N}(K_v)$. 

\begin{lemma}\label{normintersection}
Viewing $H_f^1(K_v, E^F[2])$ as sitting inside $H^1(K_v, E[2])$, we have $$H_f^1(K_v, E[2]) \cap H^1_f(K_v, E^F[2]) \simeq  E_\mathbf{N}(K_v)/2E(K_v)$$
\end{lemma}
\begin{proof}
This is Proposition 7 in \cite{KK} and Proposition 5.2 in \cite{MR2}. The proof in \cite{MR2} works even at places above $2$ and $\infty$.
\end{proof}

If $E(K)[2] \simeq \Zt$, then there is an isogeny $\phi:E\rightarrow E^\prime$ with kernel $C = E(K)[2]$ that gives rise to a $\phi$-Selmer group, $\Sel_\phi(E/K)$. There is a connecting map arising from Galois cohomology, $\delta_\phi:E^\prime(K)/\phi(E(K)) \rightarrow H^1(K, C)$,  taking the coset of $Q \in E^\prime(K)$ to the coset defined by the cocycle $c(\sigma) = \sigma(R) - R$ where  $R$ is any point on $E(\overline{K})$ with $\phi(R) = Q$. Identifying $C$ with $\mu_2$,  we can view $H^1(K, C)$ as $K^\times/(K^\times)^2$ and under this identification, $\delta_\phi(C) = \langle \Delta_E \rangle$, where $\Delta_E$ is the discriminant of (any model of) $E$. The map $\delta_\phi$ can be defined locally as well and for each place $v$ of $K$, we define a distinguished local subgroup $H^1_\phi(K_v, C)\subset  H^1(K_v, C)$ as the image of $E^\prime(K_v)/\phi(E(K_v))$ under $\delta_\phi$. We define the \textbf{$\mathbf \phi$-Selmer group of $\mathbf E$}, denoted $\Sel_\phi(E/K)$, as $$\Sel_\phi(E/K) = \ker \left ( H^1(K, C) \xrightarrow{\sum res_v} \bigoplus_{v \text{ of } K} H^1(K_v, C)/H^1_\phi(K_v, C) \right ).$$

The isogeny $\phi$ on $E$ gives gives rise to a dual isogeny $\hat \phi$ on $E^\prime$ whose kernel is $C^\prime = \phi(E[2])$. Exchanging the roles of $(E, C, \phi)$ and $(E^\prime, C^\prime, \hat \phi)$ in the above defines the $\mathbf{\hat \phi}$\textbf{-Selmer group}, $\Sel_\hatphi(E^\prime/K)$, as a subgroup of $H^1(K, C^\prime)$. The local conditions $H^1_\phi(K_v, C)$ and $H^1_\hatphi(K, C^\prime)$ are connected via the following exact sequence.

\begin{proposition}\label{localseq}
The sequence \begin{equation}\label{locseq}0 \rightarrow C^\prime/\phi \left ( E(K_v)[2]  \right ) \xrightarrow{\delta_\phi} H^1_\phi(K_v, C) \xrightarrow{i} H^1_f(K_v, E[2]) \xrightarrow{\phi} H^1_{\hat \phi}(K_v, C^\prime) \rightarrow 0\end{equation} is exact.
\end{proposition}
\begin{proof}
This well-known result follows from the sequence of kernels and cokernels arising from the composition $\hatphi \circ \phi = [2]_E$. See Remark X.4.7 in \cite{AEC} for example.
\end{proof}

The following two theorems allow us to compare the $\phi$-Selmer group, the $\hat \phi$-Selmer group, and the 2-Selmer group .

\begin{theorem}\label{gss}The $\phi$-Selmer group, the $\hat \phi$-Selmer group, and the 2-Selmer group sit inside the exact sequence \begin{equation}0 \rightarrow E^\prime(K)[2]/\phi(E(K)[2]) \xrightarrow{\delta_\phi} \Sel_\phi(E/K) \rightarrow \Sel_2(E/K) \xrightarrow{\phi}\Sel_\hatphi(E^\prime/K).\end{equation}
\end{theorem}
\begin{proof}
This is a diagram chase based on the exactness of (\ref{locseq}). See Lemma 2 in \cite{FG} for example.
\end{proof}





\begin{theorem}[Cassels]\label{prodform2}
The \textbf{Tamagawa ratio}, defined as $\T(E/E^\prime) = \frac{ \big | \Sel_\phi(E/K)  \big |}{\big |\Sel_{\hat \phi}(E^\prime/K)\big |}$, is given by a local product formula $$\mathcal{T}(E/E^\prime) = \prod_{v \text{of } K}\frac{\left | H^1_\phi(K_v, C)\right |}{2}.$$
\end{theorem}
\begin{proof}
This is a combination of Theorem 1.1 and equations (1.22) and (3.4) in \cite{Cassels8}. This product converges since $H^1_\phi(K_v, C)$ equals the unramified local subgroup $H^1_u(K_v, C)$ for all $v \nmid 2\Delta_E\infty$. \end{proof}

\section{Local Conditions for Curves in $\mathcal{F}$ }\label{localF}

The goal of this section is to prove the following proposition.

\begin{proposition}\label{allsmall}
Let $E = \En \in \mathcal{F}$. Then $\dimF H^1_\phi(K_v, C^F) \ge H^1(K_v, C)-1$ for every place $v$ of $K$, where $C^F = E^F(K)[2]$.
\end{proposition}

Let $E = \En \in \F$. The point $P = (-\frac{1}{4}, \frac{1}{8})$ on $E$ has order $2$ and $E^\prime = E/\langle P \rangle$ can be given by a model $y^2 +xy = x^3 + 64n^2x^2 +4n^2(1+256n^2)x.$ The discriminants of the model (\ref{modelEn}) for $E$ and this model for $E^\prime$ are given by $\Delta_E = 4n^2(1+256n^2)^3$ and by $\Delta_{E^{\prime}}  = 16n^4(1+256n^2)^3$ respectively.  As $1+256n^2 \not \in (K^\times)^2$,  we have $E(K)[2] = \langle P \rangle$. Since $\Delta_E$ and $\Delta_{E^\prime}$ differ by a square, we get that $K(E[2]) = K(E^\prime[2])$ and it follows that $\dimF E(K_v)[2] = \dimF E^\prime(K_v)[2]$ for every place $v$ of $K$. Proposition \ref{allsmall} will follow from some results applicable to all curves that have $K(E[2]) = K(E^\prime[2])$ and some results that are specific to curves in $\mathcal{F}$.



\begin{remark} Forthcoming work of this author shows if $E(K)[2] \simeq \Zt$, then $E$ does not have a cyclic 4-isogeny defined over $K$ but acquires one over $K(E[2])$ if and only if $K(E[2]) = K(E^\prime[2])$. See Section 4 of \cite{K} for more details. \end{remark} 

\begin{lemma}\label{addplaces}
Let $E$ be an elliptic curve with $E(K)[2] \simeq \Zt$ and suppose further that $K(E[2]) = K(E^\prime[2])$. If $E$ has additive reduction at a place $v \nmid 2$, then $\dimF H^1_\phi(K_v, C) = 1$. 
\end{lemma}
\begin{proof}
Let $E_0(K_v)$ be the group of points on $E(K_v)$ with non-singular reduction, $E_1(Kv)$ the subgroup of points with trivial reduction, and $\mathbb{F}_v$ the residue field of $K_v$. The formal group structure on $E_1(K_v)$ shows that $E_1(K_v)$ is uniquely divisible by $2$ and since $E_0(K_v)/E_1(K_v) \simeq \mathbb{F}_v^+$, $E_0(K_v)$ is uniquely 2-divisible as well. Since $E(K_v)$ has a point of order $2$, Tate's algorithm then shows that $E(K_v)/E_0(K_v)$ -- and therefore $E(K_v)[2^\infty]$ -- either injects to $\Zt \times \Zt$ or is cyclic of order 4. 

Therefore, if $E(K_v)$ has a point $R$ of order 4, then $2R \in C$. It follows that $\phi(R) \in E^\prime(K_v)[2] - C^\prime$ and $E^\prime(K_v)[2] \simeq \Zt \times \Zt$. This contradicts the fact that $\dimF E(K_v)[2] = \dimF E^\prime(K_v)[2]$ since the 2-part of $E(K_v)$ is cyclic. This shows that $E(K_v)$ can't have any points of order $4$ and similar logic shows that the same is true for $E^\prime(K_v)$. It then follows that   $\dimF E^{\prime}(K_v)/\phi( E(K_v) ) = 1$ since $\dimF E(K_v)[2] = \dimF E^\prime(K_v)[2]$ and $\phi$ has degree 2. 
\end{proof}

\begin{lemma}\label{norminter}
Let $E$ be an elliptic curve with $E(K)[2] \simeq \Zt$ and suppose that $K(E[2]) = K(E^\prime[2])$. If $E$ has split multiplicative reduction at a place $v$ where the Kodaira symbols of $E$ and $E^\prime$ are $I_{n}$  and $I_{2n}$ respectively, then $H^1_\phi(K_v, C) = H^1(K_v, C)$. \linebreak Further, if $F/K$ is a quadratic extension in which $v$ does not split, then \linebreak $\dimF H^1_\phi(K_v, C^F) = \dimF H^1(K_v, C) - 1$ and $H^1_\phi(K_v, C^F) = N_{F_w/K_v} F_w^\times / (K_v^\times)^2$, where $w$ is the place of $F$ above $v$.
\end{lemma}

\begin{proof}
Since $E$ and $E^\prime$ have split multiplicative reduction at $v$, $E/K_v$ and $E^\prime/K_v$ are $G_{K_v}$ isomorphic to Tate curves $E_q$ and $E_{q^\prime}$ respectively. By the condition on the Kodaira symbols, $|q|_v^2 = |q^\prime|_v$. Observe that $E_q$ can be two-isogenous to three different curves: $E_{q^2}, E_{\sqrt{q}}, \text{ and } E_{-\sqrt{q}}$. The curve $E_{q^\prime}$ must therefore be one of these possibilities and the only possibility with $|q|_v^2 = |q^\prime|_v$ is $q^\prime = q^2$. We therefore get $G_{K_v}$ isomorphisms $\overline{K_v}^\times/q^\Z  \rightarrow E(\overline{K_v})$ and $\overline{K_v}^\times/q^{2\Z}  \rightarrow E^\prime(\overline{K_v})$ such that the following diagram commutes.

\begin{center}\label{diag1}\leavevmode
\begin{xy} \xymatrix{
\overline{K_v}^\times/q^\Z \ar[d] \ar[r]^{x \mapsto x^2} & \overline{K_v}^\times/q^{2\Z} \ar[d] \ar[r]^{x \mapsto x} & \overline{K_v}^\times/q^\Z \ar[d]\\
E(\overline{K_v}) \ar[r]^{\phi} & E^\prime(\overline{K_v}) \ar[r]^{\hat \phi} & E(\overline{K_v})
}\end{xy}\end{center}

Because the maps in this diagram are $G_{K_v}$ equivariant, we can restrict to $K_v$ giving the following diagram, where the vertical arrows are isomorphisms.
\begin{center}\leavevmode
\begin{xy} \xymatrix{
{K_v}^\times/q^\Z \ar[d] \ar[r]^{x \mapsto x^2} & {K_v}^\times/q^{2\Z} \ar[d] \ar[r]^{x \mapsto x} & {K_v}^\times/q^\Z \ar[d]\\
E({K_v}) \ar[r]^{\phi} & E^\prime({K_v}) \ar[r]^{\hat \phi} & E({K_v})
}\end{xy}\end{center}

We therefore get a sequence of $G_K$-isomorphisms
\begin{equation*}
\resizebox{\hsize}{!}{$H^1_\phi(K_v, C) \simeq E^\prime(K_v)/\phi(E(K_v)) \simeq  ({K_v}^\times/q^{2\Z})/ ({K_v}^\times/q^\Z)^2 \simeq K_v^\times/(K_v^\times)^2 \simeq H^1(K_v, C)$}
\end{equation*}
and that $H^1_\hatphi(K_v, C^\prime) = 0$ proving the first part of the lemma.

Further, by the exactness of (\ref{locseq}), the map $i:H^1(K_v, C) \rightarrow H^1_f(K_v, E[2])$ is surjective. Because $E^\prime(K_v) \simeq {K_v}^\times/q^{2\Z}$, we see that $E^\prime(K_v)[2] = \Zt \times \Zt$. Since $K(E[2]) = K(E^\prime[2])$, we then see that $E(K_v)[2] = \Zt \times \Zt$ as well. The exactness of (\ref{locseq}) then shows that $i$ is injective. We therefore get that the restriction $\tilde i:H^1_\phi(K_v, C^F) \rightarrow H^1_f(K_v, E[2]) \cap H^1_f(K_v, E^F[2])$ is also injective.

Let $c \in H^1_f(K_v, E[2]) \cap H^1_f(K_v, E^F[2])$. As $H^1_\phi(K_v, C) = 0$, $c$ maps trivially into $H^1_\hatphi(K_v, C^{\prime F})$ under the map $\phi$ in (\ref{locseq}). It follows from Proposition \ref{localseq} that $c$ is in the image of $H^1_\phi(K_v, C^F)$ and that $\tilde i:H^1_\phi(K_v, C^F) \rightarrow H^1_f(K_v, E[2]) \cap H^1_f(K_v, E^F[2])$ is surjective. Therefore $\tilde i$ is an isomorphism.


By Lemma \ref{normintersection}, $H^1_f(K_v, E[2]) \cap H^1_f(K_v, E^F[2]) = N_{F_w/K_v}E(F_w)/2E(K_v).$
The elliptic curve norm map $N_{F_w/K_v}:E(F_w) \rightarrow E(K_v)$ translates into the usual field norm $N_{F_w/K_v}:F_w^\times/q^\Z \rightarrow K_v^\times/q^{2\Z}$, so $H^1_f(K_v, E[2]) \cap H^1_f(K_v, E^F[2])$ can be identified with $$\left ( N_{F_w/K_v} F_w^\times / q^{2Z} \right ) \big / \left ( K_v^\times / q^\Z \right )^2  \simeq  N_{F_w/K_v} F_w^\times / (K_v^\times)^2.$$ The isomorphism $E^{\prime F}(K_v)/\phi(E^F(K_v)) \rightarrow E(K_v)/2E(K_v) \cap E^F(K_v)/2E^F(K_v)$ is given by $\hat \phi$. As $\hat \phi$ is given by $x \mapsto x$ in the above diagram, the identification of $H^1_f(K_v, E[2]) \cap H^1_f(K_v, E^F[2])$ with $N_{F_w/K_v} F_w^\times / (K_v^\times)^2$ identifies $H^1_{\phi}(K_v, C^F)$ with $N_{F_w/K_v} F_w^\times / (K_v^\times)^2$. Standard results from the theory of local fields then give that $\dimF H^1_\phi(K_v, C^F) = \dimF H^1(K_v, C) -1$.
\end{proof}

\begin{lemma}\label{multcase}
If $E = \En \in \mathcal{F}$, then $E$ has multiplicative reduction at primes $\p \mid 2n$. Further, if $k = \ord_\p 2n$, then $E$ has Kodaira symbol $I_{2k}$ at $\p$ and $E^\prime$ has Kodaira symbol $I_{4k}$ at $\p$.
\end{lemma}
\begin{proof}
If $\p \mid 2n$, then the model (\ref{modelEn}) is minimal at $\p$. The reduction of (\ref{modelEn}) mod $\p$ has a node so $E$ has multiplicative reduction at $\p$. We can then read the Kodaira symbols for $E$ and $E^\prime$ at $\p$ off of the denominators of their j-invariants which are $j(E) = \frac{(1+1024n^2)^3}{4n^2}$ and $j(E^\prime) = \frac{(1+64n^2)^3}{16n^4}$ respectively.
\end{proof}



\begin{proof}[Proof of Proposition \ref{allsmall}]
Lemma \ref{multcase} combined with Lemma \ref{norminter} show that the proposition is true for all places $v\mid 2n$. The j-invariant of $E$ shows that these are the only places where $E^F$ can have multiplicative reduction and the result then follows from  Proposition \ref{addplaces}.
\end{proof}

\section{Proof of Main Theorem}\label{pfofmain}

We begin by relating $d_2(E/K)$ to the 2-adic valuation of $\mathcal{T}(E/E^\prime)$.

\begin{proposition}\label{ord2T}
If $E(K)[2] \simeq \Zt$ and $K(E[2]) = K(E^\prime[2])$, then $$d_2(E/K) \ge \ord_2 \mathcal{T}(E/E^\prime).$$
\end{proposition}
\begin{proof}
From the definition, we have \begin{equation}\label{ordsum}\ord_2 \mathcal{T}(E/E^\prime)  = \dimF \Sel_\phi(E/K) - \dimF \Sel_\hatphi(E^\prime/K).\end{equation} Since $E(K)[2] \simeq \Zt$ and $K(E[2]) = K(E^\prime[2])$, we get that $E^\prime(K)[2] \simeq \Zt$ as well. It then follows from Theorem \ref{gss} that $\dimF \Sel_\hatphi(E^\prime/K) \ge 1$ and that the map of $\Sel_\phi(E/K)$ into $\Sel_2(E/K)$ is $2$-to-$1$.  Combined with  (\ref{ordsum}), we get that the image of $\Sel_\phi(E/K)$ in $\Sel_2(E/K)$ has $\Ftwo$-dimension at least $\ord_2 \mathcal{T}(E/E^\prime)$.

Let $P$ generate $E(K)[2]$ and let $c \in \Sel_2(E/K)$ be the image of $P$ in $\Sel_2(E/K)$. We can represent $c$ by a cocycle $\hat c: G_K \rightarrow E[2]$ given by $\hat c(\gamma) = \gamma(R) - R$ for some $R \in E(\overline{K})[4]$ with $2R = P$. Observe that since $2R = P$, it must be that $\phi(R) \in E^\prime[2] - C^\prime$. If $\sigma(R) - R \in C$ for every $\sigma \in G_K$, then $\phi(R) \in E^\prime(K)$ since $\phi(C) = 0$ and $\phi(\sigma(R) - R) = \sigma(\phi(R)) - R$ for $\sigma \in G_K$. Since this would contradict $E^\prime(K)[2] \simeq \Zt$, it must be that $\sigma(R) - R \not \in C$ for some $\sigma \in G_K$ and $c$ therefore does not come from $H^1(K, C)$. We therefore get that $d_2(E/K) \ge \ord_2 \mathcal{T}(E/E^\prime).$
\end{proof}

Theorem \ref{badfamily} now follows easily from Proposition \ref{allsmall}.
\begin{proof}[Proof of Theorem \ref{badfamily}]
Let $E = \En \in \mathcal{F}$ and $F/K$ quadratic.

By Lemma \ref{prodform2}, $\ord_2 \T ( E^F/E^{\prime F} )$ is given by $$\ord_2 \T ( E^F/E^{\prime F} ) = \sum_{v \text{ of } K}  \left (  \dimF H^1_\phi(K_v, C^F) - 1 \right ).$$ By Proposition \ref{allsmall}, we get that $\dimF H^1_\phi(K_v, C^F) - 1 \ge 0$ for all places $v \nmid 2\infty$. This yields 
\begin{equation*}
\begin{split}
\ord_2 \T ( E^F/E^{\prime F} ) \ge -(r_1 + r_2) + \sum_{v \mid 2}  \left (\dimF H^1_\phi(K_v, C^F) - 1 \right ) \\
\ge  -(r_1 + r_2) +  \sum_{v \mid 2}\left(  \dimF H^1(K_v, C) - 2 \right ),
\end{split}
\end{equation*}
with the second inequality following from Proposition \ref{allsmall} as well.

As $H^1(K_v, C) \simeq K_v^\times/(K_v^\times)^2$, we get that $\dimF H^1(K_v, C) = 2 + [K_v:\Q_2]$ for places $v \mid 2$. We therefore have $$\ord_2 \T ( E^F/E^{\prime F} ) \ge -(r_1 + r_2) +\sum_{v\mid 2} [ K_v: \Q_2]  = -(r_1 + r_2) + [K:\Q] = r_2.$$
 
Proposition \ref{ord2T} then shows that $d_2(E^F/K) \ge r_2$.

The family $\F$ is infinite since every number field $K$ has infinitely many $n$ with $1+256n^2 \not \in (K^\times)^2$. The curves $E_n$ have distinct j-invariants and therefore are not isomorphic over $\overline{K}$. Since all of the $E_n$ have multiplicative reduction at all places above $2$,  work of Mazur and Rubin in \cite{MR} shows that none of them have constant 2-Selmer parity.
\end{proof}

\section*{Acknowledgements} This paper is based on work conducted by the author as part of his doctoral thesis at UC-Irvine under the direction of Karl Rubin and was supported in part by NSF grants DMS-0457481 and DMS-0757807. I would like to express my utmost gratitude to Karl Rubin for the guidance and assistance he provided while undertaking this research. I would also like to thank the reviewer for making many helpful suggestions, in particular suggesting the model (\ref{modelEn}) for $\En$.

\bibliographystyle{mrl}
\bibliography{citations}  

\end{document}